\documentclass[10pt]{amsart}
\title{Singularities of the representation variety of the braid group on 3 strands}
\author{Kevin De Laet}
\address{Department of Mathematics, University of Antwerp \\ 
 Middelheimlaan 1, B-2020 Antwerp (Belgium) \\ {\tt kevin.delaet2@uantwerpen.be}}
\date{}
\usepackage{amsmath,amsfonts,array,amscd,amsthm,amssymb,stackrel,verbatim,wrapfig,subfig,float}
\usepackage{enumerate}
\usepackage{url}
\usepackage{tikz}
\usepackage[all]{xy}

\usetikzlibrary{arrows,matrix}
\usetikzlibrary{shapes.geometric}
\usetikzlibrary{decorations.markings} % voor pijlen in midden van lijn

\tikzset{
  vertice/.style={circle,draw=black},
  decoration={markings,mark=at position 0.5 with {\arrow{>}}}
}

\theoremstyle{plain}
\newcommand{\wis}[1]{{\text{\em \usefont{OT1}{cmtt}{m}{n} #1}}}
\newcommand{\C}{\mathbb{C}}

\newcommand{\Z}{\mathbb{Z}}

\newcommand{\mA}{\mathcal{A}}

\newcommand{\V}{\mathbf{V}}
\newtheorem{theorem}{Theorem}
\newtheorem{lemma}[theorem]{Lemma}
\newtheorem{proposition}[theorem]{Proposition}
\newtheorem{corollary}[theorem]{Corollary}
\newtheorem{remark}[theorem]{Remark}

\DeclareMathOperator{\Ext}{Ext}
\DeclareMathOperator{\End}{End}
\DeclareMathOperator{\Hom}{Hom}
\DeclareMathOperator{\im}{Im}

\numberwithin{equation}{section}
\begin{document}
\begin{abstract}
The singularities of $\wis{rep}_n B_3$ are studied, where $B_3$ is the knot group on 3 strands. Specifically, we determine which semisimple representations are smooth points of $\wis{rep}_n B_3$.
\end{abstract}
\maketitle
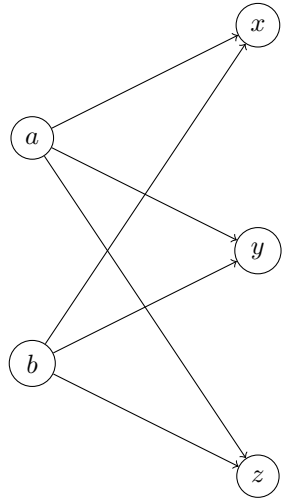
\begin{figure}[h]
\begin{center}
\begin{tikzpicture}[scale = 3]
   \node[vertice] (b1) at ( 0, 0) {$a$};
   \node[vertice] (b2) at ( 0, -1) {$b$};
   \node[vertice] (b3) at ( 1, 0.5) {$x$};
   \node[vertice] (b4) at ( 1, -0.5) {$y$};
   \node[vertice] (b5) at ( 1, -1.5) {$z$};
   \path[->] (b1) edge (b3);
   \path[->] (b1) edge (b4);
   \path[->] (b1) edge (b5);
   \path[->] (b2) edge (b3);
   \path[->] (b2) edge (b4);
   \path[->] (b2) edge (b5);
\end{tikzpicture}  \captionof{figure}{The quiver $Q$ for $\Gamma$}
	\label{pict:modulargroup}
	\end{center}
\end{figure}
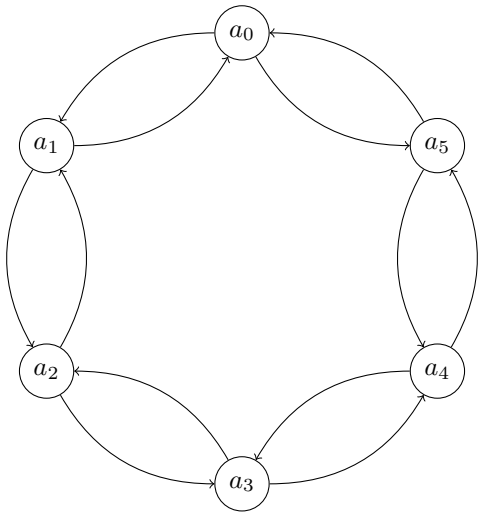
\begin{figure}[h]
\begin{center}
\begin{tikzpicture}[scale = 3]
   \node[vertice] (a0) at ( 0, 1) {$a_0$};
   \node[vertice] (a1) at ( -0.866, 0.5) {$a_1$};
   \node[vertice] (a2) at ( -0.866, -0.5) {$a_2$};
   \node[vertice] (a3) at ( 0, -1) {$a_3$};
   \node[vertice] (a4) at ( 0.866, -0.5) {$a_4$};
   \node[vertice] (a5) at ( 0.866, 0.5) {$a_5$};
   \path[->] (a0) edge[bend right=30] (a1);
   \path[->] (a1) edge[bend right=30] (a0);
   \path[->] (a1) edge[bend right=30] (a2);
   \path[->] (a2) edge[bend right=30] (a1);
   \path[->] (a2) edge[bend right=30] (a3);
   \path[->] (a3) edge[bend right=30] (a2);
   \path[->] (a3) edge[bend right=30] (a4);
   \path[->] (a4) edge[bend right=30] (a3);
   \path[->] (a4) edge[bend right=30] (a5);
   \path[->] (a5) edge[bend right=30] (a4);
   \path[->] (a5) edge[bend right=30] (a0);
   \path[->] (a0) edge[bend right=30] (a5);
\end{tikzpicture}  \captionof{figure}{The one quiver $Q'$ for $\Gamma$}
	\label{pict:onequivermodular}
	\end{center}
\end{figure}
\section{Introduction}
The braid group $B_3 = \langle \sigma_1,\sigma_2|\sigma_1 \sigma_2 \sigma_1 =\sigma_2 \sigma_1 \sigma_2  \rangle = \langle X,Y|X^2=Y^3\rangle$ on 3 strands plays an important part in knot theory, for example its representations played an important role in \cite{bruyn2011matrix} for determining knot-vertibility. In order to study the representation variety
$$
\wis{rep}_n B_3 = \{(A,B) \in \wis{GL}_n(\C) \times \wis{GL}_n(\C)| A^2=B^3\},
$$
one uses the fact that $B_3/(X^2)\cong \Z_2 * \Z_3 = \Gamma$, whose representation theory is decoded by the quiver $Q$ of Figure \ref{pict:modulargroup}, see for example \cite{bruyn2010rationality} or \cite{adriaenssens2003non}. In particular, we have
$$
\wis{rep}_n \Gamma = \{(A,B) \in \wis{GL}_n(\C) \times \wis{GL}_n(\C)| A^2=B^3=1\} = \sqcup_{\alpha} \wis{rep}_\alpha \Gamma
$$ with $\alpha=(a,b;x,y,z)$ a dimension vector for $Q$ fulfilling the requirement $a+b = x+y+z=n$. Consequently, $\wis{rep}_n \Gamma$ is smooth, such that one can compute the local quiver $Q'$ in the semisimple representation
$$
S_{1,1}^{\oplus a_0} \oplus S_{-1,\omega}^{\oplus a_1}\oplus S_{1,\omega^2}^{\oplus a_2} \oplus S_{-1,1}^{\oplus a_3}\oplus S_{1,\omega}^{\oplus a_4}\oplus S_{-1,\omega^2}^{\oplus a_5}
$$
where $S_{\rho,\tau}(\overline{X}) = \rho$, $S_{\rho,\tau}(\overline{Y}) = \tau$ and $\omega$ is a primitive 3rd root unity, which can be found in Figure \ref{pict:onequivermodular}. Questions like `which dimension vectors determine components for which the generic member is simple' and `what is the dimension of $\wis{rep}_\alpha \Gamma//\wis{PGL}_n(\C)$' can then be easily solved using the quiver technology from \cite{le2007noncommutative}.
\par However, $\wis{rep}_n B_3$ is far from smooth. Although $\wis{rep}_n B_3$ is smooth in any simple representation, in the semisimple but not simple representations singularities can occur. The main result, Theorem \ref{maintheorem}, will give a necessary and sufficient conditions for a semisimple representation to be a smooth point of $\wis{rep}_n B_3$. In addition, we show that almost all of these singularities come from the fact that they are intersection points of irreducible components of $\wis{rep}_n B_3$, except in the case that there is a 2-dimensional simple representation with multiplicity $\geq 2$ in the decomposition of the semisimple module.
\begin{remark}
Determining the singularities in semisimple points of $\wis{rep}_n B_3$ is not a full classification of all singularities of $\wis{rep}_n B_3$, as a sum of indecomposables $V$ may be smooth or singular if its semisimplification $V^{ss}$ is singular. However, if $V$ is singular, then its semisimplification $V^{ss}$ is necessarily singular, as $V^{ss} \in \overline{\mathcal{O}(V)}$ and the singular locus is a closed $\wis{PGL}_n(\C)$-subvariety of $\wis{rep}_n B_3$.
\end{remark}
\section{The connection between $B_3$ and $\Gamma$}
As $B_3$ is a central extension of $\Gamma$ with $\Z$
$$
\xymatrix{1\ar[r] & \Z \ar[r] & B_3 \ar[r]& \Gamma \ar[r] & 1}
$$
any simple representation $S$ of $B_3$ has the property that $X^2=Y^3$ acts by scalar multiplication $\lambda I_n$ on $S$ for some $\lambda \in \C^*$.
\par In addition, for each $n$, we have an action of $\C^*$ on $\wis{rep}_n B_3$ defined by
$$
t \cdot (A,B) = (t^3A,t^2B), \text{ for } t \in \C^*, (A,B) \in \wis{rep}_n B_3,
$$ 
which is of course coming from the fact that $\wis{rep}_1 B_3 = \V(x_1^2-x_2^3)\setminus\{(0,0)\} \cong \C^*$, so this action is just the same as taking the tensor product with a 1-dimensional representation of $B_3$. Using this action, we find that any simple representation of $B_3$ can be found in the $\C^*$-orbit of a simple representation of $\Gamma$. Equivalently, the corresponding map
$$
\xymatrix{\C^* \times \wis{irrep}_n \Gamma \ar[r]^-{f_n}& \wis{irrep}_n B_3}
$$
is surjective and 6-to-1, as for $t \in \pmb{\mu}_6=\{\rho \in \C^*| \rho^6=1\}$ and $(A,B) \in \wis{irrep}_n \Gamma$, $t \cdot (A,B)$ is also a simple $\Gamma$-representation. For $S \in \wis{irrep}_n \Gamma, \lambda \in \C^*$, we denote $\lambda S$ for $f_n(\lambda,S)$.
\section{Geometrical interpretation of extensions}
In here we review parts of \cite[Chapter 3 and 4]{le2007noncommutative} which we will need. Given an algebra $\mA$ and a $n$-dimensional representation $V$, we define the normal space to the $\wis{PGL}_n(\C)$-orbit of $V$ in $\wis{rep}_n \mA$ to be the vector space
$$
N_V \wis{rep}_n \mA = \frac{T_V \wis{rep}_n \mA}{T_V \mathcal{O}(V)},
$$
where $\mathcal{O}(V)$ is the $\wis{PGL}_n(\C)$-orbit of $V$. Recall that $\wis{rep}_n \mA$ is a $\wis{PGL}_n(\C)$-variety by base change, that is, if $V$ is determined by the algebra morphism
$$
\xymatrix{\mA \ar[r]^-\rho& \wis{M}_n(\C)},
$$
then $g\cdot \rho$ is the composition
$$
\xymatrix{\mA \ar[r]^-\rho& \wis{M}_n(\C)\ar[r]^-{g(-)g^{-1}}& \wis{M}_n(\C)}.
$$
In \cite[Example 3.13]{le2007noncommutative} it is shown that $N_V \wis{rep}_n \mA = \Ext^1_{\mA}(V,V)$, with for $N,M$ 2 finite dimensional $\mA$-modules with corresponding algebra morphisms $\sigma$ and $\rho$ we define
$$
\Ext^1_{\mA}(N,M)=\frac{\{ \delta \in \Hom_{\C}(\mA,\Hom_{\C}(N,M))|\forall a,a' \in A: \delta(aa') = \rho(a)\delta(a') + \delta(a)\sigma(a')\}}{\{ \delta \in \Hom_{\C}(\mA,\Hom_{\C}(N,M))|\exists \beta \in \Hom_{\C}(N,M): \forall a \in \mA: \delta(a) = \rho(a)\beta - \beta \sigma(a)\}}
$$
The top vector space is called the space of cycles and is denoted by $Z(N,M)$, while the lower vector space is the space of boundaries and is denoted by $B(N,M)$.
\par Equivalently, $\Ext^1_{\mA}(N,M)$ is the set of equivalence classes of short exact sequences of $\mA$-modules
$$
e:\xymatrix{0 \ar[r]& M\ar[r] & P \ar[r] & N \ar[r] & 0}
$$
where 2 exact sequences $(P,e)$ and $(P',e')$ are equivalent if there is an $\mA$-isomorphism $\phi \in \Hom_{\mA}(P,P')$ such that the following diagram is commutative
$$
\xymatrix{0 \ar[r]& M\ar[r] \ar[d]^-{id_M}& P \ar[r]\ar[d]^-{\phi} & N \ar[r] \ar[d]^-{id_N}& 0\\0 \ar[r]& M\ar[r] & P' \ar[r] & N \ar[r] & 0}.
$$
These 2 type of definitions of $\Ext^1_{\mA}(N,M)$ will both occur in this paper.
\section{Calculating extensions}
Following the previous section it is obvious that extensions of finite dimensional $B_3$-modules will play an important role in determining the singular points of $\wis{rep}_n B_3$.
\begin{theorem}
Let $S,T$ be irreducible representations of $\Gamma$ and let $\lambda,\mu \in \C^*$. We have
\begin{enumerate}
\item $\Ext^1_{B_3}(\lambda S, \mu T) = 0$ except if $\frac{\lambda}{\mu} \in \pmb{\mu}_6$,
\item $\Ext^1_{B_3}(\lambda S, \mu T) = \Ext^1_{\Gamma}(S,\frac{\mu}{\lambda} T)$ if $\frac{\mu}{\lambda} \in \pmb{\mu}_6$ and $S \not\cong \frac{\mu}{\lambda}T$, and lastly
\item $\Ext^1_{B_3}(\lambda S, \lambda T) = \Ext^1_{\Gamma}(S,T) \oplus \C$ if $S \cong T$.
\end{enumerate}
\label{th:exten}
\end{theorem}
\begin{proof}
For the first case, we may assume by the action of $\C^*$ that $\lambda = 1$ and $\mu \not\in \pmb{\mu}_6$. Let $\sigma$, respectively $\tau$ be the algebra morphisms from $\C B_3$ to $\End_\C(S)$, respectively $\End_\C(T)$ coming from the representations $S$ and $T$. Let $\delta$ be a cycle, that is, a linear map
$$
\xymatrix{\C B_3 \ar[r]^-\delta & \Hom_\C(S,\mu T)}
$$
such that
$$
\forall a,b \in \C B_3:\delta(ab) =\delta(a) \sigma(b) + (\mu\tau)(a)\delta(b).
$$
We need to prove that there exists a $\beta \in \Hom_\C(S,\mu T)$ such that
$$
\forall a \in \C B_3: \delta(a) =(\mu \tau)(a)\beta-\beta\sigma(a). 
$$
Using the defining relation of $B_3$, we see that 
\begin{align*}
c&=\delta(X^2) = \delta(X)\sigma(X) + \mu^3 \tau(X)\delta(X)\\
 &=\delta(Y^3) = \delta(Y)\sigma(Y)^2 + \mu^2\tau(Y)\delta(Y) \sigma(Y) + \mu^4 \tau(Y)^2\delta(Y).
\end{align*}
Assuming that $\mu \not\in \pmb{\mu}_6$, we find for $\beta = \frac{c}{\mu^6-1} \in \Hom_{\C}(S,\mu T)$
\begin{gather*}
(\mu\tau)(X)\beta - \beta\sigma(X)  \\=\frac{1}{\mu^6-1}(\mu^3\tau(X)\delta(X)\sigma(X) + \mu^6\tau(X)^2\delta(X)\\-\delta(X)\sigma(X)^2-\mu^3\tau(X)\delta(X)\sigma(X))\\
=\delta(X)
\end{gather*}
and
\begin{gather*}
(\mu\tau)(Y)\beta-\beta\sigma(Y) \\
=\frac{1}{\mu^6-1}(\mu^2\tau(Y)\delta(Y)\sigma(Y)^2+\mu^4\tau(Y)^2\delta(Y)\sigma(Y)+\mu^6\tau(Y)^3\delta(Y)\\
-\delta(Y)\sigma(Y)^3-\mu^2\tau(Y)\delta(Y)\sigma(Y)^2 - \mu^4\tau(Y)^2\delta(Y)\sigma(Y))\\
=\delta(Y).
\end{gather*}
So $\delta\subset B(S,\mu T) \subset \Hom_\C(S,\mu T)$ and consequently, $\Ext^1_{B_3}(S, \mu T) = 0$.
\par For the second case, we may again assume that $\lambda = 1$ and in addition that $\mu = 1$, for $\mu T$ is just a $\Gamma$-representation in this case. In order to prove that  $\Ext^1_{B_3}(S, T) = \Ext^1_{\Gamma}(S,T)$ for $S \not\cong T$, we have to prove that each extension of $S$ by $T$ as $B_3$-module is actually a $\Gamma$-module. Equivalently, we have to prove that for any cycle 
$$
\xymatrix{\C B_3 \ar[r]^-\delta& \Hom_\C(S,T)},
$$
we have $\delta(X^2)=0$, as we then get $\delta(X^2-1) = \delta(Y^3-1) = 0$ which makes sure that this extension is a $\Gamma$-module. We have as before
\begin{align*}
c&=\delta(X^2) = \delta(X)\sigma(X) +  \tau(X)\delta(X)\\
 &=\delta(Y^3) = \delta(Y)\sigma(Y)^2 +\tau(Y)\delta(Y) \sigma(Y) + \tau(Y)^2\delta(Y).
\end{align*}
Now we get
\begin{gather*}
c\sigma(X)-\tau(X)c \\= \delta(X)\sigma(X)^2+\tau(X)\delta(X)\sigma(X)\\-\tau(X)\delta(X)\sigma(X)-\tau(X)^2\delta(X)=0
\end{gather*}
and
\begin{gather*}
c\sigma(Y)-\tau(Y)c \\=\delta(Y)\sigma(Y)^3 +\tau(Y)\delta(Y) \sigma(Y)^2 + \tau(Y)^2\delta(Y)\sigma(Y) \\- \tau(Y)\delta(Y)\sigma(Y)^2 -\tau(Y)^2\delta(Y) \sigma(Y) - \tau(Y)^3\delta(Y)=0.
\end{gather*}
Consequently, $c \in \Hom_\Gamma(S,T)$, which is 0 if $S \not\cong T$, so $\delta$ defines a $\Gamma$-module, that is, $\delta \in \Ext^1_\Gamma(S,T)$. The inclusion $\Ext^1_\Gamma(S,T) \subset \Ext^1_{B_3}(S,T)$ is obvious.
\par For the last part, if $S \cong T$, we have for $c$ as in the second case that $c \in \End_\Gamma(S)\cong \C$. In this case, the following matrices define an element of $\Ext^1_{B_3}(S,S)$
$$X \mapsto
\begin{bmatrix}
\sigma(X)& \frac{c}{2} \sigma(X) \\ 0 & \sigma(X)
\end{bmatrix},
Y \mapsto
\begin{bmatrix}
\sigma(Y)& \frac{c}{3} \sigma(Y) \\ 0 & \sigma(Y)
\end{bmatrix}.
$$
So the function
$$
\xymatrix{\C B_3 \ar[r]^-{\delta_\sigma}& \End_\C(S)}
$$
defined by $\delta_\sigma(X) = \frac{c}{2} \sigma(X)$, $\delta_\sigma(Y) = \frac{c}{3} \sigma(Y)$ defines a derivation on $\C B_3$ with $\delta_\sigma(X^2) = c$. But then $\delta - \delta_\sigma \in \Ext^1_\Gamma(S,S)$, which shows that 
$$
\dim_\C \Ext^1_{B_3}(S,S) \leq \dim_\C \Ext^1_\Gamma(S,S)+1.
$$
Let $S \in \wis{rep}_\alpha \Gamma$ for a dimension vector $\alpha$ for $\Gamma$, then the map
$$
\xymatrix{\C^* \times \wis{irrep}_\alpha \Gamma\ar[r]^-{f_\alpha}& \wis{irrep}_n B_3}
$$
has finite fibers and is surjective on $\wis{PGL}_n(\C)$-orbits, from which it follows that 
$$
\dim_\C \Ext^1_{B_3}(S,S) \geq \dim_\C \Ext^1_\Gamma(S,S)+1,
$$
leading to the other inequality.
\end{proof}
\begin{corollary}
For all simple $B_3$-modules $S$ and $T$, we have
$$
\dim \Ext^1_{B_3}(S,T) = \dim \Ext^1_{B_3}(T,S)
$$
\end{corollary}
\begin{proof}
If $S \cong T$ the statement is trivial, so assume $S \not\cong T$. Either there exists a $\lambda \in \C^*$ such that $\lambda S$ and $\lambda T$ are $\Gamma$-modules or not. In the first case, if $S \not\cong T$ it is true because $$\Ext^1_{B_3}(S,T) = \Ext^1_{\Gamma}(\lambda S,\lambda T) = \Ext^1_{\Gamma}(\lambda T,\lambda S) = \Ext^1_{B_3}(T,S),$$ which follows from the symmetry of the one quiver for $\Gamma$. In the second case both $\Ext^1_{B_3}(S,T)$ and $\Ext^1_{B_3}(T,S)$ are 0.
\end{proof}
\section{Components of $\wis{rep}_n B_3$}
\begin{theorem}
Let $\alpha_i = (a_i,b_i;x_i,y_i,z_i)$ for $1\leq i \leq k$ be simple dimension vectors for $\Gamma$, let $n_i = a_i+b_i$ and let $n = \sum_{i=1}^k n_i$.
Then the closure of the image of the map
$$
\xymatrixcolsep{5pc}\xymatrix{\wis{PGL}_n(\C)\times^{\left(\prod_{i=1}^k \wis{GL}_{n_i}(\C)\right)/\C^*}  \left((\C^*)^k \times \prod_{i=1}^k \wis{rep}_{\alpha_i} \Gamma \right)\ar[r]^-{f_{(\alpha_1,\ldots,\alpha_k)}} & \wis{rep}_n(B_3)}
$$
is an irreducible component of $\wis{rep}_n(B_3)$.
\end{theorem}
\begin{proof}
Let $X = \wis{PGL}_n(\C)\times^{\left(\prod_{i=1}^k \wis{GL}_{n_i}(\C)\right)/\C^*} \left((\C^*)^k \times \prod_{i=1}^k \wis{rep}_{\alpha_i} \Gamma \right)$ and put $Y = \wis{rep}_n(B_3)$. On the open subset of $X$ defined by $\forall 1 \leq i \leq k: V_i \in \wis{rep}_{\alpha_i} \Gamma$ simple, the map $f_{(\alpha_1,\ldots,\alpha_k)}$ has finite fibers, so we have $\dim \overline{f_{(\alpha_1,\ldots,\alpha_k)}(X)} = \dim X$. In order to prove that $\overline{f_{(\alpha_1,\ldots,\alpha_k)}(X)}$ is indeed a component of $\wis{rep}_n(B_3)$, it is then enough to prove that there exists an element of $x \in X$ such that 
$\dim T_x X =\dim T_{f_{(\alpha_1,\ldots,\alpha_k)}(x)} Y$. Let $V = \oplus_{i=1}^k t_i V_i$ be a semisimple $B_3$-representation such that
\begin{itemize}
\item $V_i \in  \wis{rep}_{\alpha_i} \Gamma$ is a simple $\Gamma$-module and
\item $\frac{t_i}{t_j}\not\in \pmb{\mu}_6$ for all $1\leq i < j \leq k$.
\end{itemize}
Then by Theorem~\ref{th:exten} we have
\begin{align*}
\dim \Ext^1_{B_3}(V,V) &= k+\sum_{i=1}^k \dim \Ext^1_\Gamma(V_i,V_i).
\end{align*}
As both $\left((t_i)^k_1,(V_i)_{i=1}^k \right)$ and $V$ have the same $\wis{PGL}_n(\C)$-stabilizer, we are done, as we have
\begin{align*}
\dim T_{(1,(t_i)^k_1,(V_i)_{i=1}^k)}X &= \dim \mathcal{O}(1,(t_i)^k_{i=1},(V_i)_{i=1}^k) + \dim (\C^*)^k+\sum_{i=1}^k \dim \Ext^1_\Gamma(V_i,V_i)
\\ &= n^2-k+k+\sum_{i=1}^k \dim \Ext^1_\Gamma(V_i,V_i)= n^2 + \sum_{i=1}^k \dim \Ext^1_\Gamma(V_i,V_i),\\
\dim T_V Y &= \dim \mathcal{O}(V) + \Ext^1_{B_3}(V,V) 
\\ &= n^2-k + k+\sum_{i=1}^k \dim \Ext^1_\Gamma(V_i,V_i) = n^2 +\sum_{i=1}^k \dim \Ext^1_\Gamma(V_i,V_i).
\end{align*}
\end{proof}
In fact, this proof shows that the image of elements fulfilling $\frac{t_i}{t_j} \not\in \pmb{\mu}_6$ belong to the smooth locus of a component of $\wis{rep}_n B_3$.
\begin{corollary}
Let $V = \oplus_{i=1}^k S_i^{e_i}$ be a $n$-dimensional semi-simple $B_3$-module and let $S_i = \lambda_i S'_i$ with $S'_i$ a simple $\Gamma$-module, $S'_i \in \wis{rep}_{\alpha_i} \Gamma$. Then $V$ belongs to a component of dimension
$$
n^2+\sum_{i=1}^k e_i \dim \Ext^1_{\Gamma}(S'_i,S'_i).
$$
\end{corollary}
\begin{proof}
Let $S'_i \in \wis{rep}_{\alpha_i} \Gamma$, then $V$ belongs to the image of the map $f_{(\underbrace{\alpha_1,\ldots,\alpha_1}_{e_1},\ldots,\underbrace{\alpha_k,\ldots,\alpha_k}_{e_k})}$. But then the torus $(\C^*)^{\sum_{i=1}^k e_i}$ acts on $\im(f_{(\underbrace{\alpha_1,\ldots,\alpha_1}_{e_1},\ldots,\underbrace{\alpha_k,\ldots,\alpha_k}_{e_k})})$ such that for a generic element of $\mathbf{t}\in(\C^*)^{\sum_{i=1}^k e_i}$, there are no extensions between any simple subrepresentations of $\mathbf{t}V$. It follows that the dimension of this component is
\begin{align*}
\dim \overline{\im(f_{(\underbrace{\alpha_1,\ldots,\alpha_1}_{e_1},\ldots,\underbrace{\alpha_k,\ldots,\alpha_k}_{e_k})})}&=n^2-\sum_{i=1}^k e_i + \sum_{i=1}^k e_i (\dim \Ext^1_{B_3}(S_i,S_i))\\
&=
n^2+\sum_{i=1}^k e_i \dim \Ext^1_{\Gamma}(S'_i,S'_i).
\end{align*}
\end{proof}

\begin{proposition}
One of the components of $\wis{rep}_n B_3$ is isomorphic to $\wis{GL}_n(\C)$ as $\wis{PGL}_n(\C)$-variety.
\end{proposition}
\begin{proof}
Let $X$ be the irreducible component of $\wis{rep}_n B_3$ with generic element being a direct sum of 1-dimensional representations. This component can be described as
$$
X = \{(A,B)\in \wis{GL}_n(\C)\times \wis{GL}_n(\C) | A^2=B^3, AB = BA\}.
$$
Define the maps
\begin{align*}
\xymatrix{ \wis{GL}_n(\C) \ar[r]^-{f_1} & X}, \xymatrix{ X \ar[r]^-{f_2} & \wis{GL}_n(\C)}
\end{align*}
as $f_1(G) = (G^3,G^2)$ and $f_2(A,B) = AB^{-1}$. Then $f_1$ and $f_2$ are clearly $\wis{PGL}_n(\C)$-equivariant maps and $f_1 \circ f_2 = Id_X$ and $f_2 \circ f_1 = Id_{\wis{GL}_n(\C)}$.
\end{proof}
\begin{lemma}
The only components of $\wis{rep}_n \Gamma$ equal to the $\wis{PGL}_n(\C)$-orbit of a single simple representation occur in dimension 1 and are equal to the 6 1-dimensional representations of $\Gamma$.
\end{lemma}
\begin{proof}
This amount to proving that $\wis{rep}_\alpha \Gamma//\wis{PGL}_n(\C)$ has dimension $\geq 1$ if $n>1$ and $\alpha$ is a simple dimension vector. If $\alpha = (a,b;x,y,z)$ and $\mathbf{a}=(a_0,a_1,a_2,a_3,a_4,a_5)$ is a dimension vector for $Q'$ such that 
$$
\begin{cases}
a = a_0 + a_2 +a_4,\\
b = a_1 + a_3 +a_5,\\
x = a_0 + a_3\\
y = a_1 + a_4,\\
z = a_2 + a_5,
\end{cases}
$$
then the dimension of the moduli space $\wis{rep}_\alpha \Gamma//\wis{PGL}_n(\C)$ is equal to $1- \chi_{Q'}(\mathbf{a},\mathbf{a})$, with $Q'$ the Euler form on $\Z^6\times \Z^6$ associated to $Q'$, which in this case is determined by the matrix
$$
\begin{bmatrix}
1 & -1 & 0 & 0 & 0 & -1\\
-1 & 1 & -1 & 0 & 0 & 0\\
0 & -1 & 1 & -1 & 0 & 0\\
0 & 0 & -1 & 1 & -1 & 0\\
0 & 0 & 0 & -1 & 1 &-1 \\
-1 & 0 &0 & 0 & -1 & 1
\end{bmatrix},
$$
see for example \cite[Chapter 4]{le2007noncommutative}. In addition, $\wis{rep}_{\mathbf{a}} Q'$ contains simple representations of $\Gamma$ if and only if either $\mathbf{a}$ lies in the $\Z_6$-orbit of $e_1=(1,0,0,0,0,0)$ or $\mathbf{a}$ fulfils the condition
\begin{align}
\forall i \in \Z_6: a_i \leq a_{i-1}+a_{i+1},
\label{con:simp}
\end{align}
indices taken in $\Z_6$. In the case that $\wis{supp}(\mathbf{a})$ is the union of 2 consecutive vertices that fulfil condition \ref{con:simp}, the dimension vector has to lie in the $\Z_6$-orbit of $(1,1,0,0,0,0)$. For $\alpha$ itself this boils down to
\begin{align}
\max\{x,y,z\}\leq \min\{a,b\}
\label{con:simpQ}
\end{align} except if either $x$, $y$ or $z$ is equal to $0$, in which case $\alpha$ has to lie in the $\Z_6$-orbit of either $(1,0;1,0,0)$ or $(1,1;1,1,0)$. A quick calculation shows that $\chi_{Q'}(e_{1},e_{1})=1$. So assume that $\mathbf{a}$ fulfils condition \ref{con:simp} and that $\chi_{Q'}(\mathbf{a},\mathbf{a})=1$. We have
\begin{align*}
\chi_{Q'}(\mathbf{a},\mathbf{a})&= \sum_{i=0}^5 a_i^2 - \sum_{i=0}^2 2a_{2i}(a_{2i-1}+a_{2i+1})\leq \sum_{i=0}^2 a_{2i}^2-\sum_{i=0}^2 a_{2i+1}^2\\
&=\sum_{i=0}^5 a_i^2 - \sum_{i=0}^2 2a_{2i+1}(a_{2i}+a_{2(i+1)})\leq \sum_{i=0}^2 a_{2i+1}^2-\sum_{i=0}^2 a_{2i}^2
\end{align*}
which is impossible because this would imply that $1\leq 0$.
\end{proof}
In addition, we also have for representations $V \in \wis{rep}_{\mathbf{a}} Q', W\in \wis{rep}_{\mathbf{b}} Q'$ that
$$
\dim \Hom_{\C Q'}(V,W) - \dim \Ext^1_{\C Q'}(V,W) = \chi_{Q'}(\mathbf{a},\mathbf{b}),
$$
see for example \cite[Chapter 4]{le2007noncommutative}. In particular, if $V$ is simple, then  $\dim \Ext^1_{\C Q'}(V,V)$ is $0$ if and only if $\chi_{Q'}(\mathbf{a},\mathbf{a})=1$, which we have just shown is only true if $V$ is 1-dimensional.
Now we can prove the main theorem.
\begin{theorem}
Let $V = \oplus_{i=1}^k S_i^{e_i}$ be a semisimple $n$-dimensional representation of $V$ with $S_i = \lambda_i S'_i$ with $S'_i \in \wis{rep}_{\alpha_i} \Gamma$ simple. Then $\wis{rep}_n B_3$ is smooth in $V$ if and only if
\begin{itemize}
\item $\forall 1\leq i < j \leq k: \Ext^1_{B_3} (S_i,S_j) = 0$ and
\item $\forall 1\leq i \leq k: \dim S_i = 1 \text{ or } e_i = 1$.
\end{itemize}
\label{maintheorem}
\end{theorem}
\begin{proof}
We have shown that $V$ belongs to a component of dimension
$$
n^2+\sum_{i=1}^k e_i \dim \Ext^1_{\Gamma}(S'_i,S'_i).
$$
So $\wis{rep}_n B_3$ is smooth in $V$ if and only if the dimension of the tangent space in $V$ to $\wis{rep}_n B_3$ is equal to this number. Calculating the dimension of the tangent space, we find
\begin{align*}
\dim T_V \wis{rep}_n B_3  &= n^2-\sum_{i=1}^k e_i^2+\sum_{i=1}^k e_i^2 \dim \Ext^1_{B_3}(S_i,S_i) +  \sum_{1\leq i<j\leq k} 2e_ie_j \dim \Ext^1_{B_3}(S_i,S_j)\\
&=n^2+\sum_{i=1}^k e_i^2 \dim \Ext^1_{\Gamma}(S'_i,S'_i) + \sum_{1\leq i< j\leq k} 2e_ie_j \dim \Ext^1_{B_3}(S_i,S_j).
\end{align*}
So $\wis{rep}_n B_3$ is smooth in $V$ if and only if
$$
\sum_{i=1}^k e_i(e_i-1)\dim \Ext^1_{\Gamma}(S'_i,S'_i) +\sum_{1\leq i< j\leq k} 2e_ie_j \dim \Ext^1_{B_3}(S_i,S_j) = 0.
$$
As this is a sum of positive numbers, this can only be 0 if each term is 0, which already leads to the first condition of the theorem. It additionally follows that 
$$\forall 1 \leq i \leq k: e_i=1\text{ or } \dim\Ext^1_{\Gamma}(S'_i,S'_i)=0.$$
So if $e_i \neq 0$, then $\dim S'_i = 1$ in light of the previous lemma.
\end{proof}
Most of the singularities are the consequence of the following theorem.
\begin{theorem}
If $V= \oplus_{i=1}^k S_i^{e_i}$ is a singular point of $\wis{rep}_n B_3$ such that
\begin{itemize}
\item either $\exists 1\leq i < j \leq k: \Ext^1_{B_3} (S_i,S_j) \neq 0$ or
\item $\exists 1\leq i \leq k: \dim S_i \geq 3 \text{ and } e_i \geq 2$.
\end{itemize}
Then $V$ lies on the intersection of 2 components.
\end{theorem}
\begin{proof}
Let $S_i = \lambda_i S'_i$ with $S'_i \in \wis{rep}_{\alpha_i} \Gamma$ as before. In the first case, we may assume that $(i,j)=(1,2)$ and that $\lambda_1 = \lambda_2=1$. Then $V$ belongs to the image of 
$$f_{(\underbrace{\alpha_1,\ldots,\alpha_1}_{e_1},\ldots,\underbrace{\alpha_k,\ldots,\alpha_k}_{e_k})}$$
and 
$$f_{(\alpha_1+\alpha_2,\underbrace{\alpha_1,\ldots,\alpha_1}_{e_1-1},\underbrace{\alpha_2,\ldots,\alpha_2}_{e_2-1},\ldots,\underbrace{\alpha_k,\ldots,\alpha_k}_{e_k})}.$$
$\alpha_1+\alpha_2$ is also a simple dimension vector of $\Gamma$, so the claim follows.
\par In the second case, we may again assume that $i=1$ and that $S_1$ is a $\Gamma$-module. If $\dim S_1\geq 3$, then $2 \alpha_i$ also fulfils requirement \ref{con:simpQ} and $V$ lies on the intersection of the image of 
$$f_{(\underbrace{\alpha_1,\ldots,\alpha_1}_{e_1},\ldots,\underbrace{\alpha_k,\ldots,\alpha_k}_{e_k})}$$
and 
$$f_{(2\alpha_1,\underbrace{\alpha_1,\ldots,\alpha_1}_{e_1-2},\underbrace{\alpha_2,\ldots,\alpha_2}_{e_2-1},\ldots,\underbrace{\alpha_k,\ldots,\alpha_k}_{e_k})}.$$
\end{proof}
However, if there exists an $S$ of dimension 2 in the decomposition of $V$ with multiplicity $\geq 2$, then $V$ is still a singular point of $\wis{rep}_n B_3$, but $V$ is not necessarily the intersection of 2 components with generically semisimple elements.

\end{document}